\documentclass[10pt,reqno]{amsart}
\usepackage{graphicx,amssymb,mathrsfs,amsmath,color,amsthm,fancyhdr}
\usepackage{float}
\newtheorem{lem}{Lemma}
\newtheorem{rem}{Remark}

\newtheorem{conj}{Conjecture}
\newtheorem{theo}{Theorem}
\newtheorem{pro}{Proposition}

\newtheorem{dfn}{Definition}

\title[Zeros of the Bergman kernel of the Fock-Bargmann-Hartogs domain]{Zeros of the Bergman kernel of the Fock-Bargmann-Hartogs domain and the interlacing property}
\author{Atsushi Yamamori}
\address{Graduate School of Mathematics, Nagoya University, Furo-Cho, Chikusa-Ku, Nagoya 464-8602, Japan}
\begin{document}
\email{d08006u@math.nagoya-u.ac.jp}
\begin{abstract}
In this paper we consider the zeros of the Bergman kernel of
the Fock-Bargmann-Hartogs domain $\{ (z,\zeta)\in\Bbb{C}^n\times\Bbb{C}^m; ||\zeta||^2 <e^{-\mu ||z||^2}\}$ where $\mu>0$.
We describe how the existence of zeros of the Bergman kernel depends on the integers $m$ and $n$ with the help of the interlacing property.
\end{abstract}
\maketitle
\section{Introduction}
\newtheorem{Thm}{Theorem}
\def\theThm{\Alph{Thm}}
Let $\Omega$ be a domain in $\mathbb C^n$ and $K_\Omega (z,w)$ its Bergman kernel.
In \cite{Lu}, Lu Qi-Keng conjectured that if $\Omega$ is simply connected, then $K_{\Omega}$ is zero-free on $\Omega \times \Omega$.
It is already known that this conjecture is false in general (see \cite{Boas1986,S}).
A domain in $\mathbb C^n$ is called a Lu Qi-Keng domain if its Bergman kernel function is zero-free.\par
Let $\mu>0$. In our previous works \cite{Y2010,Rims}, 
we obtained an explicit formula of the Bergman kernel of $D_{n,m}= \{ (z,\zeta)\in\Bbb{C}^n\times\Bbb{C}^m; ||\zeta||^2 <e^{-\mu ||z||^2}\}$ 
which is called the Fock-Bargmann-Hartogs domain (abbr. FBH domain) in this paper.
The aim of this paper is to establish the following theorem:
\begin{Thm}
 For any fixed $n\in \mathbb N$, there exists a unique number $m_0(n)\in \mathbb{N}$ such that
$D_{n,m}$ is a Lu Qi-Keng domain if and only if $m\geq m_0(n)$.
\end{Thm}
As a related work, we should mention the following result proved by L. Zhang and W. Yin in \cite{Zhang2009}.
\begin{Thm}[\mbox{\cite[Theorem 1]{Zhang2009}}]\label{zhang}
For fixed $n$ and $p$, there exists a constant $m_0=m_0(n,p)$ such that 
$$\Omega_{n,m}^{p,1}:= \{ (z,\zeta) \in\mathbb{C}^n \times \mathbb C^m ; ||\zeta||^{2p} + || z||^2 <1\}$$
is a Lu Qi-Keng domain for all $m\geq m_0$.
\end{Thm}
In spite of apparent  similarity between Theorems A and B,
there are some differences.
First, the number $m_0$ in Theorem B depends on $n$ and $p$,
while $m_0$ in Theorem A depends only on $n$.
Second, Theorem B does not describe whether or not $\Omega_{n,m}^{p,1}$ is a Lu Qi-Keng domain for $m< m_0$.
On the other hand, Theorem A states that $D_{n,m}$ is not a Lu Qi-Keng domain for $m< m_0$.
Moreover the sequence $\{m_0(n)\}_{n=1}^\infty$ is monotonically increasing.
In other words, if $D_{n,m}$ is not a Lu Qi-Keng domain, neither is $D_{n+1,m}$ (Theorem \ref{mono}).\par
The organization of this paper is as follows. We review basic notions and results in Section 2, which will be used for the proof of our main theorems.
In Section 4, we present remaining problems.
\section{Preliminaries}
\subsection{Interlace polynomial}
\begin{dfn}\label{def:int}
 Let $f$ and $g$ be polynomials with only real roots.
We denote the roots of $f(\mbox{resp. }g) $ by $a_1, \cdots, a_n   (\mbox{resp. }b_1, \cdots, b_m )$ , where $a_1 \leq \cdots \leq a_n
(\mbox{resp. } b_1 \leq \cdots \leq b_m )$.
We say that $g$ alternates $f$ if $\deg f=\deg g=n$ and
\begin{eqnarray}\label{int}
 b_1 \leq a_1 \leq b_2 \leq a_2\cdots \leq b_n \leq a_n .
\end{eqnarray}
We say that $g$ interlaces $f$ if $\deg f=\deg g+1=n$ and
\begin{eqnarray}\label{int2}
  a_1 \leq b_1 \leq a_2\cdots \leq b_{n-1} \leq a_n .
\end{eqnarray}
Let $g \preccurlyeq f$ denote that either $g$ alternates $f$ or $g$ interlaces $f$. 
If no equality sign occurs in $(\ref{int})$ $(respectively (\ref{int2}))$
then we say that $g$ strictly alternates $f$ $($respectively $g$ strictly interlaces $f$ $)$.
Let $g  \prec f$ denote that either $g$ strictly alternates $f$ or $g$ strictly interlaces $f$.
\end{dfn}
As a simple consequence of Rolle's Theorem, we have the following lemma.
\begin{lem}\label{roll}
If $f$ is a polynomial with only real roots and all roots are distinct, then $f'\prec f$.
\end{lem}
A real polynomial is said to be standard if either it is identically zero or its leading coefficient is positive.
Let $RZ$ denote the set of real polynomials with only real zeros.\par
L. L. Liu and Y. Wang \cite{Liu} proved the following results which play substantial roles in the proofs of our theorems.
\begin{lem}[\mbox{\cite[Lemma 2.5]{Liu}}]\label{liu}
Let $G(x) = c(x)f (x) + d(x)g(x)$ where $G, f, g$ are standard and $c, d$ are
real polynomials. Suppose that $f, g \in RZ$ and $g \prec f$. Then the following holds.\\
$(i)$ If $\deg G \leq \deg g + 1$ and $c(s) > 0$ whenever $g(s) = 0$, then $G \in RZ$ and $g \prec G$.\\
$(ii)$ If $\deg G \leq \deg f$ and $d(r) > 0$ whenever $f (r) = 0$, then $G \in RZ$ and $G \prec f$.\\
The statements also hold if all instances of $\prec$ and $\succ$ are replaced by $\preceq$
and $\succeq$
respectively.
\end{lem}
\begin{theo}[\mbox{\cite[Theorem 2.1]{Liu}}]\label{liu2}
Let $F,f,g$ be three real polynomials satisfying the following conditions.\\
(a) $F(x)= a(x)f(x)+ b(x)g(x)$, where $a,b$ are two real polynomials, such that $\deg F = \deg f$ or $\deg f +1$.\\
(b) $ f,g \in RZ$ and $g\preccurlyeq f$.\\
(c) $F$ and $g$ have leading coefficients of the same sign.\\
Suppose that $b(r)\leq 0$ whenever $f(r)=0$. Then $F\in RZ$ and $f\preceq F$. In particular, if $g\prec f$ and $b(r)<0$ whenever $f(r)=0$,
then $f \prec F$.
\end{theo} 
\subsection{Polylogarithm function}
The logarithm $\log t$ is obtained as the analytic continuation of the formula
$$-\log (1-t) = \sum_{k=1}^\infty \frac{t^k}{k}, \quad |t|<1 $$
to $\mathbb C^*$.
The polylogarithm function $Li_s (t)$ is a natural generalization of the right hand side: 
\begin{eqnarray}\label{def:poly}
 Li_s(t)=\sum_{k=1}^\infty \frac{t^k}{k^s}.
 \end{eqnarray}
It converges for $|t|<1$ and any $s\in\Bbb C$.
If $s$ is a negative integer, say $s=-n$, then the polylogarithm function has
the following closed form:
\begin{eqnarray*}
  Li_{-n}(t)&=
\dfrac{ t }{(1-t)^{n+1}}
 \sum_{j=0}^{n-1}  A(n,j+1)t^{j},
\end{eqnarray*}
where $A(n,m)$ is the Eulerian number \cite[eq.(2.17)]{d2010}\begin{eqnarray*}
 A(n,m)=\sum_{\ell=0}^{m} (-1)^\ell \binom{n+1}{\ell} (m-\ell)^n.
\end{eqnarray*}
The first few are
$$\begin{array}{ll}
 Li_{-1}(t)=\dfrac{t}{(1-t)^2}, &  Li_{-2}(t)=\dfrac{t^2 + t}{(1-t)^3}, \\
 Li_{-3}(t)=\dfrac{t^3 + 4t^2+ t}{(1-t)^4}, &  Li_{-4}(t)=\dfrac{t^4 +11t^3 +11t^2 +t}{(1-t)^5}.
\end{array}$$ 
The polynomial $A_n(t) = \sum_{j=0}^{n-1}  A(n,j+1)t^{j}$ is called the Eulerian
polynomial.
We give here known properties which are used later.
\begin{pro}\label{well}
$(i)$ $A_n(t)$ has only real negative simple roots.\\
$(ii)$ $Li_{-n}(t)/t$ has a zero $t_0$ such that $|t_0|<1$
for all $n \geq 3$ (see \cite{Y2010}).
\end{pro}
More information about the polylogarithm function and the Eulerian polynomial can be found in \cite{comtet,d2010,H}.
\subsection{Bergman kernel}
Let $\Omega$ be a domain in $\Bbb{C}^n$, $p$ a positive continuous function on $\Omega$ and $L^2_a(\Omega,p)$
the Hilbert space of square integrable holomorphic
functions with respect to the weight function $p$ on $\Omega$ with the inner
product 
$$ \langle f,g \rangle=\int _{\Omega} f(z)\overline{g(z)} p(z) dz,\mbox{\quad
for all $f,g\in \mathcal O(\Omega)$. }$$
The weighted Bergman kernel $K_{\Omega, p}$ of $\Omega$ with respect to the
weight $p$ is the
reproducing kernel of $L^2_a(\Omega,p)$. If $p\equiv 1$, the reproducing kernel is called the  Bergman kernel.\par
Define the Hartogs domain by
$ \Omega_{m,p}:=  \{ (z,\zeta)\in\Omega\times\Bbb{C}^m; ||\zeta||^2 <
p(z)\}$.
E. Ligocka \cite[Proposition 0]{Ligocka1989} showed that
the Bergman kernel of $\Omega_{m,p}$ is expressed as infinite sum of
weighted Bergman kernels of the base domain $\Omega$.
\begin{theo}\label{ligocka}
 Let $K_m$ be the Bergman kernel of $\Omega_ {m ,p }$ and $K_{\Omega,p^k}$
the
weighted Bergman kernel of
$\Omega$ with respect to the weight function $p^k$. Then
\begin{eqnarray*}
 K_m((z,\zeta),(z',\zeta') ) =\dfrac{m!}{\pi^m}\sum_{k=0}^\infty
\dfrac{ (m+1)_k}{k!}
K_{\Omega,p^{k+m}}(z,z')\langle\zeta,\zeta'\rangle^k.
\end{eqnarray*}
Here $(a)_k$ denotes the Pochhammer symbol
$(a)_k =a(a+1)\cdots (a+k-1).  $
\end{theo}
In our previous paper \cite{Y2010} we obtained a formula of the Bergman kernel of FBH domain with the help of Theorem \ref{ligocka}.
\begin{theo}[\mbox{\cite{Y2010}}]
The Bergman kernel $ K_{D_{n,m}}$ of FBH domain is given by
\begin{eqnarray}
 K_{D_{n,m}}((z,\zeta),(z',\zeta') ) =
\dfrac{\mu^n}{\pi^{n+m}}
e^{m\mu \langle z,z'\rangle }\dfrac{d^m}{d t^m} Li_{-n}(t)\lvert_{t=e^{\mu\langle 
z,z'\rangle}\langle \zeta ,\zeta'  \rangle  }%\\
\end{eqnarray}
\end{theo}
\section{Lu Qi-Keng problem of $D_{n,m}$}
In \cite{Y2010}, we showed that the Lu Qi-Keng problem of $D_{n,m}$ is reduce to study of zeros of $\frac{d^m}{dt^m}Li_{-n}(t)$:
\begin{pro}\label{pro}
The domain $D_{n,m}$ is a Lu Qi-Keng domain if and only if $\frac{d^m}{dt^m}Li_{-n}(t)$ has no zeros in $|t|<1$. 
\end{pro}
The objective of this section is to describe how the zeros of $\frac{d^m}{dt^m}Li_{-n}(t)$ depend on integers $n$ and $m$.\par
Define a polynomial $A_{n,m}(t)$ by the following equation:
\begin{eqnarray}\label{def:anm}
\dfrac{d^m}{dt^m}Li_{-n}(t) =\dfrac{A_{n,m}(t) }{(1-t)^{n+m+1}}.
\end{eqnarray}
\begin{rem}\label{closed}
There is a closed expression of $A_{n,m}(t)$ as follows:
\begin{eqnarray}
A_{n,m}(t) = m!\sum_{j=0}^{n} (-1)^{n+j}  (m+1)_j S(1+n,1+j)(1-t)^{n-j},
\end{eqnarray}
where $S(\cdot,\cdot)$ denotes the Stirling number of the second kind.
This formula is a simple consequence of the formula \cite[eq. 2.10c]{d2010}:
$$Li_{-n} (t) = \sum_{j=0}^{n}\frac{ (-1)^{n+j}  j!
S(1+n,1+j)   }{(1-t)^{j+1}}  .$$
\end{rem}
\begin{lem}\label{rec}
(i) The polynomial $A_{n,m}(t)$ satisfies the recurrence relation
\begin{eqnarray}\label{zenkasiki}
A_{n,m+1}(t)=(n+m+1)A_{n,m}(t)+ (1-t)A_{n,m}'(t)
\end{eqnarray}
with the initial condition $A_{n,1}(t)=A_{n+1}(t)$. \\
(ii) All coefficients of $A_{n,m}(t)$ are positive.
\end{lem}
\begin{proof}
(i) If we differentiate both sides of equation (\ref{def:anm}), then we have
\begin{eqnarray}
\dfrac{d^{m+1}}{dt ^{m+1}} Li_{-n} (t) = \dfrac{(n+m+1)A_{n,m}(t) + (1-t)A_{n,m}'(t)}{(1-t)^{n+m+2}},
\end{eqnarray}
which proves the recurrence relation (\ref{zenkasiki}).
The initial condition is verified from the formula $\frac{d}{dt} Li_{-n}(t)=Li_{-n-1}(t)/t$. \\ 
(ii) Fix $n\in\mathbb{N}$.
We shall use the induction on $m$.
Since the Eulerian number $A(n,k)$ is equal to the number of permutations of $n$ objects with $k-1$ rises (see \cite[p242]{comtet}),
the coefficients of the Eulerian polynomial $A_{n,1}(t) =A_{n+1}(t)$ are all positive.\par
Assume that all coefficients of $A_{n,m} (t)$ are positive.
Put $A_{n,m}(t)=\sum_{i=0}^n a_i t^i$ where $a_i>0$ for all $0\leq i \leq n$.
By (\ref{zenkasiki}), the coefficients of $A_{n,m+1}(t)=\sum_{i=0}^n a' _i t^i$ are expressed as follows:
$$ a'_i =
\begin{cases}
(n+m+1)a_0 +a_1, &\mbox{if $i=0$,}\\
(m+1)a_n, &\mbox{if $i=n$,}\\
(n+m+1-i)a_i + (i+1)a_{i+1}, &\mbox{otherwise.}
\end{cases}
$$
Thus $a'_i$ is positive if $a_i>0$ for each $i$.
This completes the proof of the statement (ii).
\end{proof}
For the proof of our main theorem, we quote the following general result.
\begin{lem}[\mbox{\cite[Corollary 1.2.3]{Book}}]\label{bound}
Suppose $p(t):=a_n t^n + \cdots a_1 t + a_0$ with $a_k >0$ for each $k$.
Then all zeros of $p$ lie in the annulus
$$\min_{0 \leq i \leq n-1 } \{ a_i/ a_{i+1}\} \leq |t|  \leq \max_{0 \leq i \leq n-1} \{a_i/ a_{i+1}\}.$$
\end{lem}
Now we are ready to state our main theorem.
\begin{theo}\label{main}
(i) For any $n,m\in\mathbb N$, we have $A_{n,m}(t)\succ A_{n,m+1}(t)$.\\
(ii) For any fixed $n\in \mathbb N$, there exists a unique number $m_0(n)\in \mathbb{N}$ such that
$D_{n,m}$ is a Lu Qi-Keng domain if and only if $m\geq m_0(n)$.
\end{theo}
\begin{proof}
(i)
Fix $n\in\mathbb N$. We prove the theorem by induction on $m$.
Let us first prove the statement for $m=1$.
From Proposition \ref{well}, we know that the Eulerian polynomial $A_{n,1}(t)=A_{n+1}(t)$ has only negative real simple roots.
Combining this fact and Lemma \ref{roll}, we have $A_{n,1} (t) \succ A'_{n,1} (t)$.
Thus the polynomial $G(t)=A_{n,2}(t)$ satisfies the conditions of Lemma \ref{liu}(ii)
with $f(t)=A_{n,1} (t)$, $g(t)=A'_{n,1} (t)$, $c(t)=n+2$ and $d(t)=1-t$. 
Actually, they satisfy the assumption of Lemma \ref{liu}(ii) because $1-r>0$ whenever $A_{n,1}(r)=0$.
Hence $A_{n,2}(t)\in RZ$ and $A_{n,1}\succ A_{n,2}$.
We have proved the statement for $m=1$.\par
Assume $A_{n,m}\succ A_{n,m+1}$. By the definition of interlace, we see that $ A_{n,m+1}(t)$ has only simple roots.
Moreover, from Lemma \ref{rec}, 
all coefficients of $A_{n,m+1}(t)$ are positive, so that all roots of $A_{n,m+1}(t)$ are negative.
Hence the polynomial $G(t)=A_{n,m+2}(t)$ satisfies the conditions of Lemma \ref{liu}(ii)
with $f(t)=A_{n,m+1} (t)$, $g(t)=A'_{n,m+1} (t)$, $c(t)=n+m+2$ and $d(t)=1-t$. Hence $A_{n,m+1}\succ A_{n,m+2}$.
We have thus proved the statement (i).\\
(ii)
Denote the largest root of $A_{n,m}(t)$ by $r_{n,m}$.
Then the definition of interlace implies
$0>r_{n,m}>r_{n,m+1}$.
Let us show $r_{n,m}\rightarrow-\infty$ as $m\rightarrow \infty$.
Put $\widetilde A_{n,m}(t)= A_{n,m}(t+1)/m! =\sum_{j=0}^{n}   (m+1)_j
S(1+n,1+j) t^{n-j}$.
From Remark \ref{closed} and Lemma \ref{bound}, it is enough to show $\min \{ a_i/ a_{i+1}: 1\leq i <n \} \rightarrow \infty$ as $m \rightarrow \infty$ for $\widetilde A_{n,m}(t)$.
It can be shown by simple computation that $a_i/ a_{i+1}=(m+n-i)S(1+n,n-i+1)/S(1+n, n-i)$, which is
a linear polynomial of $m$ with the positive leading coefficient. 
Therefore $r_{n,m}\rightarrow-\infty$ as $m \rightarrow \infty$.
We set $ m_0(n)= \min\{m\in\mathbb N; r_{n,m} \leq -1 \}$.
If $m\geq m_0(n) $ then $r_{n,m}  \leq-1$, so that $D_{n,m}$ is a Lu Qi-Keng domain by Proposition \ref{pro}.
On the other hand, if $m<m_0(n)$ then $r_{n,m} >-1$, and $D_{n,m}$ is not a Lu Qi-Keng domain.
Hence the statement (ii) is proved.
\end{proof}
In Theorem \ref{main}, we have proved $A_{n,m}(t)\succ A_{n,m+1}(t)$.
It is natural to expect similar relation between $A_{n,m}(t)$ and $A_{n+1,m}(t)$.
Indeed, we can prove the following theorem.
\begin{theo}\label{main2}
For any $n,m\in\mathbb N$, we have $A_{n,m}(t) \prec A_{n+1,m}(t)$.
\end{theo}
For the proof of Theorem \ref{main2}, we need the following lemma.
\begin{lem}\label{lem:rec}
For any $n,m\in\mathbb{N}$, we have
$$A_{n+1,m}(t)= t A_{n,m+1}(t) +m(1-t)A_{n,m}(t)  .$$
\end{lem}
\begin{proof}
 We see from (\ref{def:poly}) that the $m$-th derivative of the polylogarithm function has the following series representation:
\begin{eqnarray}
F_{n,m}(t) =\dfrac{d^m Li_{-n} (t) }{d t^m} =  \sum_{k=0}^\infty (k+1)_m
(k+m)^{n}t^k, 
\end{eqnarray}
for $|t|<1$.
Then it is easy to see
\begin{eqnarray}\label{rec2}
F_{n+1,m}(t)= mF_{n,m}(t)+ t F_{n,m+1}(t).
\end{eqnarray}
The relation (\ref{rec2}) and the definition of $A_{n,m}$ (see (\ref{def:anm})) imply
$$A_{n+1,m}(t)= m(1-t)A_{n,m}(t)+ t A_{n,m+1}(t) .$$
\end{proof}
\begin{proof}[Proof of Theorem \ref{main2}.]
We already know that $A_{n,m}(t)$ has only negative roots.
From Theorem \ref{main} (i), Lemma \ref{lem:rec} and this fact, we see that
the polynomial $G(t)=A_{n+1,m}(t)$ satisfies the conditions of Theorem \ref{liu2} with
$f(t)=A_{n,m}(t), g(t)=A_{n,m+1}(t), a(t)=m(1-t),b(t)=t$. In particular we have $g \prec f$ and $b(r)<0$ whenever $f(r)=0$.
Therefore we finally obtain $A_{n,m}(t) \prec A_{n+1,m}(t)$.
\end{proof}
We now obtain the following theorem:
\begin{theo}\label{mono}
The sequence  $\{m_0(n)\}_{n=1}^\infty$ is monotonically increasing.
In other words, if $D_{n,m}$ is not a Lu Qi-Keng domain, neither is $D_{n+1,m}$.
\end{theo}
\begin{proof}
If $D_{n,m}$ is not a Lu Qi-Keng domain, then we have $r_{n,m}>-1$.
Theorem \ref{main2} implies that $-1< r_{n,m} <r_{n+1,m}<0 $.
Hence $D_{n+1,m}$ is not a Lu Qi-Keng domain.
\end{proof}
\section{Remaining problems}
Some numerical computation of $m_0(n)$ by Mathematica indicates the following conjecture.
\begin{conj}
The sequence $\{m_0(n)\}_{n=1}^\infty$ is
strictly monotonically increasing.
\end{conj}
\begin{table}[H]
\begin{center}
\begin{tabular}{|c|c|c|c|c|c|c|c|c|c|c|c|c|c|c|c|}
\hline
$n$ & 1 & 2 & 3 & 4 & 5 & 6 & 7 & 8 &9 &10&11&12&13&14&15\\
\hline 
$m_0(n)$ & 1 & 3 &6& 8&11&14& 17 &20 &23& 26&29&32&35&38&42\\
\hline
\end{tabular}
\caption{The values of $m_0(n) $}
\end{center}
\end{table}
Our sequence $\{m_0(n)\}_{n=1}^\infty$ and the sequence A050503 in the On-Line Encyclopedia of Integer Sequences
suggest the following conjecture.
\begin{conj}
$m_0(n)\leq [ (n+1) \log(n+1) ]$,
where $ [x]$ denotes the nearest integer of $x$. In particular, equality holds if and only if $n\leq 10$.
\end{conj}
Put $f(n)=[(n+1) \log(n+1)]$. Then the first few of $f(n)$ are given as follows.
\begin{table}[H]
\begin{center}
\begin{tabular}{|c|c|c|c|c|c|c|c|c|c|c|c|c|c|c|c| }
\hline
$n$ & 1 & 2 & 3 & 4 & 5 & 6 & 7 & 8 &9 &10&11&12&13&14&15\\
\hline 
$f(n)$  & 1 &3& 6&8&11& 14 &17 &20& 23&26&30&33&37&41&44\\
\hline
\end{tabular}
\caption{The values of $ f(n)$}
\end{center}
\end{table}
In the proof of Theorem \ref{main}, we showed  $\lim_{m\rightarrow \infty}r_{n,m}=-\infty.$
On the other hand, we do not succeed in computing the value $\lim_{n\rightarrow \infty}r_{n,m}$.
Some numerical computation indicate the following conjecture.
\begin{conj}\label{rmn}
 $\lim_{n\rightarrow \infty}r_{n,m}=0$
\end{conj}
If we admit that the Conjecture \ref{rmn} is true, we have the following:
\begin{conj}
For any fixed $ m\in \mathbb{N}$, there exists a unique number $n_0(m)\in \mathbb{N}$ such that 
$D_{n,m}$ is not a Lu Qi-Keng domain if and only if $n\geq n_0(m)$.
\end{conj}
\section*{Acknowledgement}
I am indebted to Professor Hideyuki Ishi who reads the
manuscript and suggested many improvements.
\bibliographystyle{plain}

\end{document}